\documentclass{amsart} 

\usepackage{amsmath,amssymb,amsthm} 
\usepackage{graphicx} 
\usepackage[arrow, matrix, curve]{xy} 
\usepackage{color} 
\usepackage{hyperref} 
\usepackage{epsfig}
\usepackage[utf8]{inputenc}

\DeclareMathOperator{\tb}{tb}
\DeclareMathOperator{\rot}{rot}
\DeclareMathOperator{\lk}{lk}

\DeclareMathOperator{\de}{d}
\DeclareMathOperator{\e}{e}

\DeclareMathOperator{\cs}{cs}

\newcommand{\Z}{\mathbb{Z}}

\newcommand{\N}{\mathbb{N}}

\newcommand{\xist}{\xi_{\mathrm{st}}}
\newcommand{\wst}{\omega_{\mathrm{st}}}

\newtheorem{Theorem}{Theorem}[section]
\newtheorem{thm}[Theorem]{Theorem}
\newtheorem{lem}[Theorem]{Lemma}
\newtheorem{prop}[Theorem]{Proposition}
\newtheorem{cor}[Theorem]{Corollary}
\newtheorem*{Theorem-ohne}{Theorem}

\newtheorem{ques}[Theorem]{Question}

\theoremstyle{definition}

\newtheorem{rem}[Theorem]{Remark}

\begingroup\expandafter\expandafter\expandafter\endgroup
\expandafter\ifx\csname pdfsuppresswarningpagegroup\endcsname\relax
\else
\fi

\begin{document}


\title[Contact surgery graphs]{Contact surgery graphs}

\author{Marc Kegel}
\address{Humboldt-Universit\"at zu Berlin, Rudower Chaussee 25, 12489 Berlin, Germany.}
\email{kegemarc@math.hu-berlin.de, kegelmarc87@gmail.com}

\author{Sinem Onaran}
\address{Department of Mathematics, Hacettepe University, 06800 Beytepe-Ankara, Turkey.}
\email{sonaran@hacettepe.edu.tr}

\date{\today}


\begin{abstract}
We define a graph encoding the structure of contact surgery on contact $3$-manifolds and analyze its basic properties and some of its interesting subgraphs.
\end{abstract}

\makeatletter
\@namedef{subjclassname@2020}{%
  \textup{2020} Mathematics Subject Classification}
\makeatother



\maketitle


\section{Introduction}

In an unpublished work, William Thurston defined a graph consisting of a vertex $v_M$ for every diffeomorphism type of a closed, orientable $3$-manifold $M$. Two vertices $v_M$ and $v_{M'}$ are connected by an edge if there exist a Dehn surgery between $M$ and $M'$. In~\cite{HW15}, this  graph is called the \textit{big Dehn surgery graph} and  studied in various ways.

In this paper we define a directed graph encoding the structure of contact $(\pm1)$-surgeries on contact $3$-manifolds. The \textit{contact surgery graph} $\Gamma$ is the graph consisting of a vertex $v_{(M,\xi)}$ for every contactomorphism type of a contact $3$-manifold $(M,\xi)$. Whenever there exists a Legendrian knot $L$ in $(M_1,\xi_1)$ such that contact $(-1)$-surgery along $L$ yields $(M_2,\xi_2)$ we introduce a directed edge pointing from $v_{(M_1,\xi_1)}$ to $v_{(M_2,\xi_2)}$.

The contact surgery graph is closely tied to the properties of Stein or Weinstein cobordisms and fillings of contact manifolds since contact $(-1)$-surgery along a Legendrian knot in a contact $3$-manifold $(M,\xi)$, also known as Legendrian surgery, corresponds to the attachment of a Weinstein $2$-handle to the symplectization of $(M,\xi)$. The inverse operation of a contact $(-1)$-surgery is called contact $(+1)$-surgery. 

\subsection{Properties of the contact surgery graph}
In Section~\ref{sec:properties} we study the basic properties of $\Gamma$. First, we observe that the contact surgery graph $\Gamma$ is connected by the work of Ding--Geiges~\cite{DG04} who showed that any contact $3$-manifold can be constructed from the standard tight contact structure $\xist$ on $S^3$ by a sequence of $(\pm 1)$-contact surgeries. On the other hand, we know from~\cite{El90,Go98} that a contact manifold $(M,\xi)$ is Stein fillable if and only if there exists a directed path from $v_{(\#_k S^1\times S^2,\xist)}$ to $v_{(M,\xi)}$ and thus $\Gamma$ is not strongly connected since there exists non-Stein fillable contact manifolds. (Recall that an oriented graph is strongly connected if for any pair of vertices $(v_1,v_2)$ there exist a path from $v_1$ to $v_2$ following the orientations of the edges.) Etnyre--Honda~\cite{EH02} showed that there exists a directed path from any vertex corresponding to a given overtwisted contact manifold to any other vertex, i.e. $\Gamma$ is strongly connected from any vertex corresponding to an overtwisted contact manifold. The question if there exists a vertex such that $\Gamma$ is strongly connected to that vertex is equivalent to the open question if there exists a maximal element with respect to the Stein cobordism relation~\cite{We14}. 

We show that $\Gamma$ stays connected after removing an arbitrary finite collection of vertices and edges. Recall that a graph is called $k$-\textit{connected} if it is still connected after removing $k$ arbitrary vertices and $k$-\textit{edge-connected} if it remains connected after removing $k$ arbitrary edges.

\begin{thm}\label{thm:thmk-connected}
	The contact surgery graph $\Gamma$ is $k$-connected and $k$-edge-connected for any integer $k\geq0$. 
\end{thm}

We equip $\Gamma$ with its graph metric $\textrm{d}$. Then the distance from $v_{(S^3,\xist)}$ to another vertex $v_{(M,\xi)}$ equals the contact $(\pm1)$-surgery number $\cs_{\pm1}(M,\xi)$, i.e. the minimal number of components of a contact $(\pm1)$-surgery link $L$ in $(S^3,\xist)$ describing $(M,\xi)$~\cite{EKO22}.

\begin{prop}\label{prop:diameterDegree}
The contact surgery graph $\Gamma$ has infinite diameter, infinite indegree and infinite outdegree.
\end{prop}

Next, we study the existence of Euler- and Hamiltonian walks and paths. We first recall the necessary definitions. A \textit{track} $t$ in $\Gamma$ is an infinite sequence whose terms are alternately vertices and edges of $\Gamma$ starting at a vertex and such that any edge in $t$ joins the vertices preceding and following the edge in $t$. A \textit{Hamiltonian walk} of $\Gamma$ is a track running through any vertex of $\Gamma$ at least once. A \textit{Hamiltonian path (Eulerian path)} of $\Gamma$ is a track running through any vertex (edge) of $\Gamma$ exactly once. If a track is following the direction of the oriented edges it is called \textit{ditrack} and then the definitions of Hamiltonian diwalks, and Hamiltonian- and Eulerian dipaths are obvious.

\begin{thm} \label{thm:Eulerpaths}The contact surgery graph $\Gamma$ admits Eulerian and Hamiltonian paths and also Hamiltonian walks. On the other hand, there exists no Hamiltonian diwalk, no Hamiltonian dipath and no Eulerian dipath in $\Gamma$.
\end{thm}

\subsection{Contact geometric subgraphs of \texorpdfstring{$\Gamma$}{Gamma}}

In Section~\ref{sec.contact} we study interesting contact geometric subgraphs of $\Gamma$. We define the subgraphs $\Gamma_{\textrm{OT}}$, $\Gamma_{\textrm{tight}}$, $\Gamma_{\textrm{Stein}}$, $\Gamma_{\textrm{strong}}$, $\Gamma_{\textrm{weak}}$ and $\Gamma_{\textrm{c}\neq0}$ consisting of vertices (and the corresponding edges connecting two such vertices) of $\Gamma$ representing contact manifolds which are overtwisted, tight, Stein fillable, strongly fillable, weakly fillable or with a non-vanishing contact class $\textrm{c}$ in Heegaard Floer homology. We analyze some of the basic properties of these subgraphs and in particular prove the following results.

\begin{thm}\label{thm:sub}
	$\Gamma_{\textrm{OT}}$ is strongly connected. Each of $\Gamma_{\textrm{tight}}$, $\Gamma_{\textrm{Stein}}$, $\Gamma_{\textrm{strong}}$, $\Gamma_{\textrm{weak}}$ and $\Gamma_{\textrm{c}\neq0}$ is connected, but not strongly connected.
\end{thm}

\begin{thm}\label{thm:OTdiwalk}
	There exists no Hamiltonian diwalk in each of $\Gamma_{\textrm{tight}}$, $\Gamma_{\textrm{Stein}}$, $\Gamma_{\textrm{strong}}$, $\Gamma_{\textrm{weak}}$ and $\Gamma_{\textrm{c}\neq0}$ and thus also no Hamiltonian- or Eulerian dipath. On the contrary, $\Gamma_{\textrm{OT}}$ admits an Eulerian dipath and thus also a Hamiltonian diwalk.
\end{thm}

\subsection{Topological subgraphs of \texorpdfstring{$\Gamma$}{Gamma}}

In Section~\ref{sec:top} we concentrate on topological subgraphs of the contact surgery graph. Let $\Gamma_M$ denote the subgraph that consists of vertices corresponding to contact manifolds where the underlying topological manifolds are all diffeomorphic to a fixed manifold $M$. Similarly we denote by $\Gamma_{(M,\mathfrak s)}$ the subgraph of $\Gamma_M$ consisting of all contact structures on a fixed manifold $M$ lying in the same $spin^c$ structure $\mathfrak s$.

\begin{thm}\label{thm:GammaM}
	The connected components of $\Gamma_M$ are given by $\Gamma_{(M,\mathfrak s)}$, and thus the connected components of $\Gamma_M$ are in bijection with $H_1(M)$.
\end{thm}

The \textit{link} $\lk(M,\xi)$ of a contact $3$-manifold $(M,\xi)$ is  defined as
\begin{equation*}
\lk(M,\xi):=\lk(v_{(M,\xi)}):=\big\{ v_{(N,\eta)} \big| \textrm{d}(v_{(M,\xi)}, v_{(N,\eta)})=1\big\}.
\end{equation*}

In~\cite{HW15} it is shown that the link of $S^3$ in the topological surgery graph is connected and of bounded diameter. The main question still remains open: Is the link of any topological $3$-manifold connected?  It turns out that we can answer this question for contact $3$-manifolds.

\begin{thm}\label{thm:lk}
	The link $\lk(M,\xi)$ of any contact $3$-manifold $(M,\xi)$ is connected and of diameter less than $4$.
\end{thm}

\subsection*{Conventions} 

Throughout, this paper we work in the smooth category. We assume all $3$-manifolds to be connected, closed, oriented, and smooth; all contact structures are positive and coorientable. For background on contact surgery and symplectic and Stein cobordisms we refer to~\cite{GS99,DG04,DGS04,Ge08,We14,DK16,Ke17,CEK21,EKO22}. Legendrian links in $(S^3,\xist)$ are always presented in their front projection. We choose the normalization of the $\de_3$-invariant as in~\cite{CEK21,EKO22} which differs from the normalizations in~\cite{Go98,DGS04,DK16} by $1/2$. Using our normalization we see that contact structures on homology spheres have integral $\de_3$-invariants (in particular $\de_3(S^3,\xist)=0$) and that the $\de_3$-invariant is additive under connected sums.

\subsection*{Acknowledgments} 
M.K. thanks Chris Wendl and Felix Schm\"aschke for useful discussions. We would also like to thank the \textit{Mathematisches Forschungsinstitut Oberwolfach} where a part of this project was carried out when M.K. was \textit{Oberwolfach Research Fellow} in August 2020. S.O.~was partially supported by T\"UB{\.I}TAK 1001-119F411.

\section{Properties of the contact surgery graph}\label{sec:properties}

We start by discussing the basic properties of the contact surgery graph $\Gamma$.

\begin{proof}[Proof of Proposition~\ref{prop:diameterDegree}]  $\Gamma$ has infinite diameter since a single surgery can change the rank of the first homology at most by one.

To show that the outdegree of $\Gamma$ is infinite, let $v_{(M,\xi)}$ be a vertex of $\Gamma$. For an even integer $n$, choose a Legendrian knot $L$ in a Darboux-ball in $(M,\xi)$ with Thurston--Bennequin invariant $\tb=1$ and rotation number $\rot=n$. Denote the contact manifold obtained by contact $(-1)$-surgery along $L$ by $L(-1)$. A calculation as for example in~\cite{DGS04,DK16} shows that the homology of $L(-1)$ is
\begin{equation*}
H_1\big(L(-1)\big)=H_1(M)\oplus \Z_{\mu_L},
\end{equation*}
where the $\Z$-summand is generated by a meridian $\mu_L$ of $L$ and that the Poincar\'e dual of the Euler class $\e(L(-1))$ is given by
\begin{equation*}
\e(L(-1))=\e(M,\xi)+n\mu_L,
\end{equation*}
where $\e(M,\xi)$ denotes the Poincar\'e dual of the Euler class of $(M,\xi)$. Thus, we get infinitely many different contact manifolds by contact $(-1)$-surgery from $M$. The same construction with contact $(+1)$-surgery along a Legendrian knot with Thurston--Bennequin invariant $\tb = -1$ provides the fact that the indegree is also infinite.
\end{proof}

\begin{proof} [Proof of Theorem~\ref{thm:thmk-connected}] Let $E$ be a set consisting of $k$-different vertices in $\Gamma$ and let $(M_0,\xi_0)$ and $(M_n,\xi_n)$ be two contact $3$-manifolds representing vertices in $\Gamma\setminus E$. We consider a path 
\begin{equation*}
(M_0,\xi_0)\rightarrow (M_1,\xi)\rightarrow (M_2,\xi_2)\rightarrow\cdots\rightarrow (M_n,\xi_n)
\end{equation*}
in $\Gamma$ between $(M_0,\xi_0)$ and $(M_n,\xi_n)$. We denote by $L_0$ a Legendrian surgery link in $(S^3,\xist)$ representing $(M_0,\xi_0)$. Let $K_{i+1}$ be a Legendrian knot in $(M_i,\xi_i)$ such that contact $(\pm1)$-surgery along $K_i$ is contactomorphic to $(M_{i+1},\xi_{i+1})$. By Lemma 4.7.1 in~\cite{Ke17} we can represent $K_1$ as a Legendrian knot in the complement of $L_0$ and thus $L_1:=L\cup K_1$ is a surgery link in $(S^3,\xist)$ for $(M_1,\xi_1)$. By induction we get surgery links $L_i$ in $(S^3,\xist)$ representing $(M_i,\xi_i)$ describing the above path in $\Gamma$. We will construct a path from $(M_0,\xi_0)$ to $(M_n,\xi_n)$ in $\Gamma\setminus E$.

Let $p\geq2$ be an integer such that no contact structure on $M_i\#L(p,1)$ is in $E$ for all $i=0,\ldots,n$. (Such a $p$ exists because $E$ is finite.) We define a new sequence of Legendrian surgery links as follows. We set $L'_0=L_0$ and $L'_{i+1}=L_i\sqcup U_p(-1)$ where $\sqcup$ denotes the disjoint union of knots and $U_p(-1)$ denotes the contact $(-1)$-surgery along a Legendrian unknot with $\tb=1-p$. It follows that $L'_{i+1}$ represents a contact structure on $M_i\#L(p,1)$. Finally we set $L'_{n+2}$ to be the disjoint union of $L_n$ with $U_p(-1)$ together with a $(+1)$-framed meridian $\mu_{U_p}$ of $U_p$. Since $U_p(-1)\cup\mu_{U_p}(+1)$ yields $(S^3,\xist)$ by~\cite{Av13}, we see that $L'_{n+2}$ represents $(M_n,\xi_n)$.  Thus, a path in $\Gamma\setminus E$ between $(M_0,\xi_0)$ and $(M_n,\xi_n)$ is constructed. The same argument shows that $\Gamma$ is also $k$-edge connected.
\end{proof}

In fact, the above proof directly implies the following corollary.

\begin{cor}\label{cor:k-connected}
	Let $E$ be a finite set of vertices in $\Gamma$ and let $(M_1,\xi_1)$ and $(M_2,\xi_2)$ be two vertices in $\Gamma\setminus E$. Then their distances in the corresponding graphs are related as
	\begin{equation*}
	d_{\Gamma\setminus E}(M_1,M_2)\leq d_{\Gamma}(M_1,M_2)+2.
	\end{equation*}
\end{cor} 

With the above results we study the existence of Eulerian and Hamiltonian paths and walks.

\begin{proof} [Proof of Theorem~\ref{thm:Eulerpaths}] Since $\Gamma$ is connected, $k$-connected and $k$-edge-connected, for any natural number $k$, and of infinite degree the main results from~\cite{EGV36,EGV38,Na71} immediately imply that $\Gamma$ contains Eulerian and Hamiltonian paths.
	
Since contact $(-1)$-surgery preserves tightness by~\cite{Wa15}, a Hamiltonian diwalk has to start at an overtwisted contact manifold and has to run first through all overtwisted contact manifolds before reaching to a tight contact manifold. But since there exists infinitely many overtwisted contact manifolds this is not possible.
\end{proof}

Using the main result of~\cite{Na66} we conclude the following.

\begin{cor}\label{cor:biased}
	$\Gamma$ is \textbf{biased}, i.e. there exists a subset $X$ of the vertices of $\Gamma$ such that there are infinitely many edges oriented from $X$ to its complement $X^c$, but only finitely many edges oriented from $X^c$ to $X$.
\end{cor}

It would be interesting to find such a set $X$ explicitly.

\begin{proof} [Proof of Corollary~\ref{cor:biased}] The main result of~\cite{Na66} says that an oriented graph admits an Eulerian dipath based at a vertex $v$ if and only if the graph is countable, connected, $1$-coherent, $v$-solenoidal and unbiased. We refer to~\cite{Na66} for the definitions. Since the contact surgery graph $\Gamma$ is connected and countable but admits no Eulerian dipath it is enough to check that $\Gamma$ is $1$-coherent and $v$-solenoidal. $1$-coherency is a condition of the underlying unoriented graph and since $\Gamma$ admits an undirected Euler path, $\Gamma$ is $1$-coherent. Finally, it follows that $\Gamma$ is $v$-solenoidal for any vertex since the indegree and outdegree are infinite for any vertex.
\end{proof}



\section{Contact geometric subgraphs of \texorpdfstring{$\Gamma$}{Gamma}}\label{sec.contact}
Here we study the subgraphs $\Gamma_{\textrm{OT}}$, $\Gamma_{\textrm{tight}}$, $\Gamma_{\textrm{Stein}}$, $\Gamma_{\textrm{strong}}$, $\Gamma_{\textrm{weak}}$ and $\Gamma_{\textrm{c}\neq0}$.

\begin{proof}[Proof of Theorem~\ref{thm:sub}] We first show that $\Gamma_{\textrm{OT}}$ is strongly connected. Let $(M_1,\xi_1)$ and $(M_2,\xi_2)$ be two overtwisted contact manifolds. By~\cite{EH02}, there exist a directed path $p$ from $(M_1,\xi_1)$ to $(M_2,\xi_2)$ in $\Gamma$. We  argue that any vertex in $p$ corresponds to an overtwisted contact manifold. Let us assume the contrary. Since $(M_2,\xi_2)$ is overtwisted, there exists an overtwisted contact manifold $(M_{\textrm{OT}},\xi_{\textrm{OT}})$ that can be obtained by contact $(-1)$-surgery from a tight contact manifold $(M_{\textrm{tight}},\xi_{\textrm{tight}})$ contradicting Wand's result which says that contact $(-1)$-surgery preserves tightness~\cite{Wa15}. 

Next, we consider $\Gamma_*$ with $*=\textrm{tight}$, $\textrm{Stein}$, $\textrm{strong}$, $\textrm{weak}$ or $\textrm{c}\neq0$. To show that $\Gamma_*$ is connected we show that there exist an undirected path in $\Gamma_*$ from $(S^3,\xist)$ to any other contact manifold $(M,\xi)$ with property $*$. We consider the contact $(+1)$-surgery along the Legendrian unknot with Thurston--Bennequin invariant $\tb=-2$ and rotation number $\rot=1$ in $(S^3,\xist)$. It is well-known that the resulting contact manifold is the overtwisted contact structure $\xi_1$ on $S^3$ with normalized $d_3$-invariant equal to $1$~\cite{DGS04}. Then by~\cite{EH02}, there exists a directed path in $\Gamma$ of contact $(-1)$-surgeries from $(S^3,\xi_1)$ to $(M,\xi)$. However, this path runs through at least one overtwisted contact manifold and hence it is not in $\Gamma_*$. In total we get a surgery link $L$ in $(S^3,\xist)$ describing $(M,\xi)$ with only a single contact surgery coefficient $(+1)$. Now we change the order of the surgeries and first perform all contact $(-1)$-surgeries and at the end we perform the single contact $(+1)$-surgery. Since contact $(-1)$-surgery is known to preserve any of the properties $*$ by~\cite{Wa15,El90,We91,OS05}, we get a path in $\Gamma_*$ from $(S^3,\xist)$ to $(M,\xi)$.

The contact surgery subgraphs $\Gamma_{\textrm{tight}}$, $\Gamma_{\textrm{strong}}$, $\Gamma_{\textrm{weak}}$ and $\Gamma_{\textrm{c}\neq0}$ are not strongly connected since there exists in each of this graphs a contact manifold $(M,\xi)$ which is not Stein fillable~\cite{El96,Gh08} and then there cannot be a directed path from $(S^3,\xist)$ to $(M,\xi)$.

Finally we show that $\Gamma_{\textrm{Stein}}$ is not strongly connected. We show that there exists no directed path of contact $(-1)$-surgeries from $(S^3,\xist)$ to $(S^1\times S^2,\xist)$. Let us assume the contrary. Then we get a simply-connected Stein cobordism from $(S^3,\xist)$ to $(S^1\times S^2,\xist)$. We glue this Stein cobordism to the standard $4$-ball filling of $(S^3,\xist)$ to get a simply connected Stein filling $(W,\wst)$ of $(S^1\times S^2,\xist)$. However, by~\cite{El90} any Stein filling of $(S^1\times S^2,\xist)$ is diffeomorphic to $S^1\times D^3$ which is not simply-connected.
\end{proof}

\begin{rem}
	We remark that we have shown that $\Gamma_{OT}$ is a strong connected component of $\Gamma$ and we wonder what the other strong connected components of $\Gamma$ are. Using Wendl's theorem on symplectic fillings of planar contact manifolds~\cite{We10} Plamenevskaya~\cite{Pl12} deduced that any planar Stein fillable contact manifold $(M,\xi)$ cannot be obtained from itself by a sequence of contact $(-1)$-surgeries. It follows that any planar Stein fillable contact manifold is its own strong connected component.
\end{rem}

\begin{ques}
	Is $\Gamma_{OT}$ the only non-trivial strong connected component of $\Gamma$?
\end{ques} 

\begin{rem}
	We know that each of $\Gamma_{\textrm{OT}}$, $\Gamma_{\textrm{tight}}$, $\Gamma_{\textrm{Stein}}$, $\Gamma_{\textrm{strong}}$, $\Gamma_{\textrm{weak}}$ and $\Gamma_{\textrm{c}\neq0}$ has infinite diameter and infinite in- and outdegree. That $\Gamma_{\textrm{tight}}$ and $\Gamma_{\textrm{c}\neq0}$ have infinite indegree follows from the work of Lisca--Stipsicz~\cite{LS04}. As part of their main theorem they describe infinitely many different Legendrian knots such that contact $(+1)$-surgery on them yield tight contact manifolds with non-vanishing contact class. That $\Gamma_{\textrm{strong}}$, $\Gamma_{\textrm{weak}}$, and $\Gamma_{\textrm{Stein}}$ have infinite indegree follows similarly from~\cite{CET17}. The other statements follow from the arguments in the proof of Proposition~\ref{prop:diameterDegree}.
\end{rem}

For the proof of Theorem~\ref{thm:OTdiwalk} we need the following lemma.

\begin{lem}\label{lem:reversingOT}
	Let $(M,\xi^M)$ be an overtwisted manifold and $(M,\xi^M_{\textrm{stab}})$ be its stabilization, i.e. $(M,\xi^M_{\textrm{stab}})=(M,\xi^M)\#(S^3,\xi_1)$. Let $(N,\xi^N)$ be an overtwisted contact manifold which can be obtained from $(M,\xi^M_{\textrm{stab}})$ by a single contact $(-1)$-surgery. Then we can obtain $(M,\xi^M)$ by a contact $(-1)$-surgery from $(N,\xi^N)$.
\end{lem}

\begin{proof}[Proof of Lemma~\ref{lem:reversingOT}]
	Let $L$ be a Legendrian knot in $(M,\xi^M_{\textrm{stab}})$ such that $L(-1)=(N,\xi^N)$. Let $L^*$ be the dual surgery knot of $L$ in $(N,\xi^N)$. By the cancellation lemma, $L^*(+1)$ is again contactomorphic to $(M,\xi^M_{\textrm{stab}})$. Now we choose a loose Legendrian realization $L'$ of $L^*$ such that if we stabilize $L'$ once positive and once negative we get a Legendrian knot which is formally isotopic to $L^*$. (This is possible since we can destabilize any loose Legendrian knot.) 
	
	We claim that $L'(-1)$ yields $(M,\xi^M)$. Since $L'$ is topologically isotopic to $L^*$ and its contact framing and the contact framing of $L^*$ differ by 2, the contact $(-1)$-surgery along $L'$ yields topologically the same manifold as the contact $(+1)$-surgery along $L^*$ which is in fact $M$. Note that  $L'(-1)$ is overtwisted since $L'$ is a loose knot. Then a straightforward computation in a local model as in~\cite{DGS05} shows that the homotopical invariants of the contact structures agree. Thus by Eliashberg's classification of contact structures~\cite{El89} it follows that the contact structures are contactomorphic.
\end{proof}

	\begin{proof} [Proof of Theorem~\ref{thm:OTdiwalk}] 
	First we remark that by following the proof of Theorem~\ref{thm:thmk-connected},  we conclude that each of the subgraphs $\Gamma_{\textrm{OT}}$, $\Gamma_{\textrm{tight}}$, $\Gamma_{\textrm{Stein}}$, $\Gamma_{\textrm{strong}}$, $\Gamma_{\textrm{weak}}$ and $\Gamma_{\textrm{c}\neq0}$ is $k$-connected and $k$-edge-connected for any natural number $k$. Thus,  each of these subgraphs admits Eulerian and Hamiltonian paths and also Hamiltonian walks. For the directed paths we conclude as in the proof of Theorem~\ref{thm:Eulerpaths} that $\Gamma_*$ does not admit a Hamiltonian diwalk, a Hamiltonian dipath and a Eulerian dipath for $*=\textrm{tight}$, $\textrm{Stein}$, $\textrm{strong}$, $\textrm{weak}$ or $\textrm{c}\neq0$.
	
	On the other hand, we show that an Eulerian dipath exists in $\Gamma_{\textrm{OT}}$. As in the proof of Corollary~\ref{cor:biased} it follows that $\Gamma_{\textrm{OT}}$ is $1$-coherent and $v$-solenoidal. By the main result of~\cite{Na66}, it is therefore enough to show that $\Gamma_{\textrm{OT}}$ is unbiased. Assume there exists a subset $X$ of the vertices of $\Gamma_{\textrm{OT}}$ such that there are infinitely many edges oriented from $X$ to its complement $X^c$, but only finitely many edges oriented from $X^c$ to $X$. As a first step, we show that there exists for any $i\in\N$ a Legendrian knot $K_i$ in an overtwisted contact manifolds $(M_i,\xi_i)$ in $X$, such that all $M_i$ are pairwise non-diffeomorphic and the contact $(-1)$-surgery $K_i(-1)$ along $K_i$ lies in $X^c$. (This step will not use the overtwistedness.)
	
	By assumption we know that there are infinitely many edges pointing out of $X$. For any $i\in\N$, we choose some Legendrian knot $K'_i$ in an overtwisted contact manifold $(M,\xi)$ in $X$ such that $K'_i(-1)$ lies in $X^c$. We show that we can obtain infinitely many of the $K'_i(-1)$ by a contact $(-1)$-surgery from infinitely many different manifolds in $X$. For that we first observe, that we can get an overtwisted contact structure $\xi_{\textrm{OT}}$ on $M\# L(i+1,1)$ by a single contact $(+1)$-surgery along a Legendrian unknot $U_i$ in a Darboux ball in $M\setminus K_i$. Since only finitely many edges point into $X$ we conclude that an infinite subset of the $(M\# L(i+1,1),\xi_{\textrm{OT}})$ are elements of $X$ again. Performing a contact $(-1)$-surgery along $K_i$ in $(M\# L(i+1,1),\xi_{\textrm{OT}})$ yields $K_i(-1)\#L(i+1,1)$ with an overtwisted contact structure. And canceling the contact $(+1)$-surgery along $U_i$ by a contact $(-1)$-surgery along a push-off of $U_i$ yields $K_i(-1)$. Thus we can conclude that there exists an infinite family of Legendrian knots $K_i$ in different overtwisted contact manifolds $(M_i,\xi_i)$ in $X$ such that $K_i(-1)$ are elements of $X^c$.
	
	In the next step, which only works for overtwisted contact manifolds, we show that we can get back from the $K_i(-1)$ to $X$ by contact $(-1)$-surgeries. We write $(M_i,\xi_i)$ as $(M,\xi_i')\#(S^3,\xi_1)$. By Lemma~\ref{lem:reversingOT} there exists a contact $(-1)$-surgery along a Legendrian knot in $K_i(-1)$ yielding $(M_i,\xi_i')$ and from $(M_i,\xi_i')$ we can get back to $(M_i,\xi_i)$ by another contact $(-1)$-surgery. Thus we have constructed an infinite family of edges pointing from $X^c$ into $X$ contradicting the assumption and finishing the proof of Theorem~\ref{thm:OTdiwalk}.
\end{proof}

Next, we study the difference of the distance functions.

\begin{thm}
Given two overtwisted contact manifolds $(M_1,\xi_1)$ and $(M_2,\xi_2)$. \\
(1) The distance between $(M_1,\xi_1)$ and $(M_2,\xi_2)$ in $\Gamma_{\textrm{OT}}$ is at most $2$ larger than their distance in $\Gamma$.\\
(2) The minimal lengths of directed paths from $(M_1,\xi_1)$ to $(M_2,\xi_2)$ agree in $\Gamma$ and $\Gamma_{\textrm{OT}}$.
\end{thm}

\begin{proof} (1) Let $p$ be a minimal (undirected) path between $(M_1,\xi_1)$ and $(M_2,\xi_2)$ in $\Gamma$. If every vertex in $p$ corresponds to an overtwisted contact manifold the distances in $\Gamma$ and $\Gamma_{\textrm{OT}}$ agree. In general, however, the path $p$ might run through tight contact manifolds. To prevent this we choose an ordered surgery link $L$ in $(M_1,\xi_1)$ such that contact $(\pm1)$-surgery in that given order along $L$ corresponds to the path $p$. We denote by $\xi_0$ the overtwisted contact structure on $S^3$ with vanishing normalized $\de_3$-invariant. A $2$-component surgery diagram of $(S^3,\xi_0)$ is shown in Figure~\ref{fig:S3} (center). We add this surgery diagram in a Darboux ball in the exterior of $L$ to the surgery link, where we first perform the surgeries along the two new Legendrian knots and afterwards the surgeries corresponding to $p$. Then any contact manifold in the path $p$ is replaced by a connected sum with $(S^3,\xi_0)$. 
By Eliashberg's classification of overtwisted contact structures~\cite{El89}, it follows that we have constructed a new path $p'$ from $(M_1,\xi_1)$ to $(M_2,\xi_2)$ in $\Gamma_{\textrm{OT}}$ of length $2$ larger than the length of $p$.\\
(2) By Wand's theorem~\cite{Wa15}, any directed path in $\Gamma$ between two overtwisted contact manifolds cannot run through a tight contact manifold, see proof of Theorem~\ref{thm:sub}.
\end{proof}

\begin{figure}[htbp] 
\centering
\def\svgwidth{\columnwidth}
\begingroup%
  \makeatletter%
  \providecommand\color[2][]{%
    \errmessage{(Inkscape) Color is used for the text in Inkscape, but the package 'color.sty' is not loaded}%
    \renewcommand\color[2][]{}%
  }%
  \providecommand\transparent[1]{%
    \errmessage{(Inkscape) Transparency is used (non-zero) for the text in Inkscape, but the package 'transparent.sty' is not loaded}%
    \renewcommand\transparent[1]{}%
  }%
  \providecommand\rotatebox[2]{#2}%
  \newcommand*\fsize{\dimexpr\f@size pt\relax}%
  \newcommand*\lineheight[1]{\fontsize{\fsize}{#1\fsize}\selectfont}%
  \ifx\svgwidth\undefined%
    \setlength{\unitlength}{491.94372176bp}%
    \ifx\svgscale\undefined%
      \relax%
    \else%
      \setlength{\unitlength}{\unitlength * \real{\svgscale}}%
    \fi%
  \else%
    \setlength{\unitlength}{\svgwidth}%
  \fi%
  \global\let\svgwidth\undefined%
  \global\let\svgscale\undefined%
  \makeatother%
  \begin{picture}(1,0.22305288)%
    \lineheight{1}%
    \setlength\tabcolsep{0pt}%
    \put(0,0){\includegraphics[width=\unitlength,page=1]{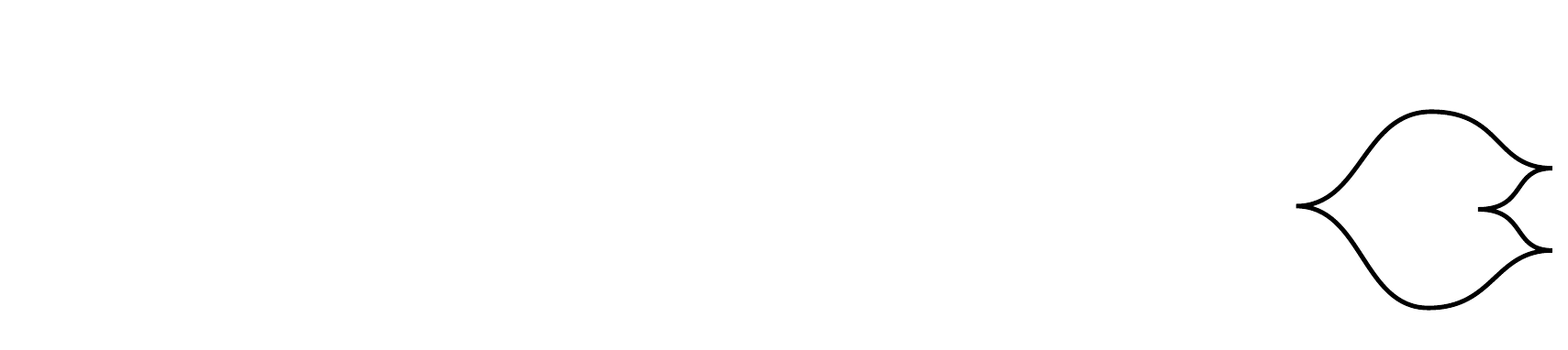}}%
    \put(0.94385818,0.14944507){\color[rgb]{0,0,0}\makebox(0,0)[lt]{\lineheight{1.25}\smash{\begin{tabular}[t]{l}$+1$\end{tabular}}}}%
    \put(0,0){\includegraphics[width=\unitlength,page=2]{S3csg.pdf}}%
    \put(0.50591194,0.15880708){\color[rgb]{0,0,0}\makebox(0,0)[lt]{\lineheight{1.25}\smash{\begin{tabular}[t]{l}$+1$\end{tabular}}}}%
    \put(0.61968705,0.18314049){\color[rgb]{0,0,0}\makebox(0,0)[lt]{\lineheight{1.25}\smash{\begin{tabular}[t]{l}$+1$\end{tabular}}}}%
    \put(0,0){\includegraphics[width=\unitlength,page=3]{S3csg.pdf}}%
    \put(0.06640797,0.16441847){\color[rgb]{0,0,0}\makebox(0,0)[lt]{\lineheight{1.25}\smash{\begin{tabular}[t]{l}$+1$\end{tabular}}}}%
    \put(0.21015276,0.12924489){\color[rgb]{0,0,0}\makebox(0,0)[lt]{\lineheight{1.25}\smash{\begin{tabular}[t]{l}$-1$\end{tabular}}}}%
  \end{picture}%
\endgroup%

\caption{Contact surgery diagrams of contact structures on $S^3$. Left: $(S^3,\xi_{-1})$, center: $(S^3,\xi_{0})$, right: $(S^3,\xi_{1})$~\cite{DGS04,EKO22}.} 
\label{fig:S3}
\end{figure}

\section{Topological subgraphs of \texorpdfstring{$\Gamma$}{Gamma}}\label{sec:top}

Before we prove Theorem~\ref{thm:GammaM} we first discuss the corresponding result for $S^3$.

\begin{lem}
$\Gamma_{S^3}$ is connected.
\end{lem}

\begin{proof} Recall that by the work of Eliashberg $\xist$ is the unique tight contact structure on $S^3$~\cite{El92} and that the overtwisted contact structures are in one-to-one correspondence to the homotopy classes of tangential $2$-plane fields~\cite{El89}, which are on homology spheres in bijection with the integers via their normalized $d_3$-invariants~\cite{Go98}. We denote the unique overtwisted contact structure on $S^3$ with normalized $d_3$-invariant equal to $n$ by $\xi_n$. 

Contact surgery diagrams for all contact structures on $S^3$ where explicitly described in~\cite{DGS04}. A contact surgery diagram of $\xi_1$ is given by the contact $(+1)$-surgery along the Legendrian unknot with Thurston--Bennequin invariant $t=-2$ and rotation number $r=1$ shown on the right of Figure~\ref{fig:S3}. The contact $(\pm1)$-surgery diagram along the $2$-component link shown on the left of Figure~\ref{fig:S3} represents $\xi_{-1}$. The disjoint union of two surgery diagrams describes a connected sum of the underlying contact manifolds. Since the $d_3$-invariant behaves additively under the connected sum, we get contact surgery diagrams of all contact structures on $S^3$ by taking appropriate disjoint unions of the contact surgery diagrams of $\xi_1$ and $\xi_{-1}$.

It follows that we can get $(S^3,\xi_1)$ by a single contact $(+1)$-surgery from $(S^3,\xist)$ and that there exists a contact $(+1)$-surgery on $(S^3,\xi_k)$ yielding $(S^3,\xi_{k+1})$ and thus conversely a contact $(-1)$-surgery from $(S^3,\xi_k)$ to $(S^3,\xi_{k-1})$.
\end{proof}

\begin{proof}[Proof of Theorem~\ref{thm:GammaM}]
	First we discuss the case that $M$ is a homology sphere. Then we can get any overtwisted contact structure on $M$ from a fixed contact structure on $M$ by connected summing with overtwisted contact structures of $S^3$~\cite{DGS04}. It follows that any two overtwisted contact structures on $M$ can be connected in $\Gamma_M$. Now let $\xi_{\textrm{tight}}$ be some tight contact structure on $M$. Then there there exists a single contact $(+1)$-surgery along a Legendrian unknot in $(M,\xi_{\textrm{tight}})$ yielding $(M,\xi_{\textrm{tight}})\# (S^3,\xi_1)$ which is overtwisted and thus it follows that $\Gamma_M$ is connected.
	
	If the underlying manifold is not a homology sphere it gets slightly more complicated since the classification of tangential $2$-plane fields is more involved. As a further invariant we have the $spin^c$ structure of a contact structure. However, it is known that for a given contact structure with $spin^c$ structure $\mathfrak s$ we can get any other overtwisted contact structure with the same $spin^c$ structure $\mathfrak s$ by connected summing with the overtwisted contact structures on $S^3$~\cite{DGS04}. Thus, we can apply the same argument as in the homology sphere case to deduce that $\Gamma_{(M,\mathfrak s)}$ is connected.
	
	It remains to show that there is no edge connecting two different $spin^c$ structures on $M$. For that we use Gompf's $\Gamma$-invariant which classifies $spin^c$ structures~\cite{Go98}. Let $\mathfrak s$ be a $spin^c$ structure on $M$ and $\xi$ be a contact structure inducing $\mathfrak s$. Let $L_0$ in $(M,\xi)$ be a Legendrian knot such that contact $(+1)$- or contact $(-1)$-surgery along $L_0$ yields another contact structure $\xi'$ on $M$. We want to show that $\xi'$ induces the same $spin^c$ structure $\mathfrak s$. For that we choose a spin structure $\mathfrak t$, describe $(M,\xi)$ by a contact surgery diagram along a Legendrian link $L=L_1\cup\ldots\cup L_n$ in $(S^3,\xist)$, and present $L_0$ as a knot in the exterior of $L$. To show that $\xi$ and $\xi'$ induce the same $spin^c$ structures it is enough to compute that Gompf's $\Gamma$ invariants of $\xi$ and $\xi'$ with respect to $\mathfrak t$ agree. We present $\mathfrak t$ via a characteristic sublink $(L_j)_{j\in J}$, $J\subset\{1,\ldots,n\}$, of $L$. Since the homologies of $M$ and $L_0(\pm1)$ agree we deduce that $\mu_{L_0}$ is nullhomologous in $L_0(\pm1)$ and that $J$ is also a characteristic sublink of the surgery diagram $L_0\cup L$ of $L_0(\pm1)$. Then we can use the formula for computing Gompf's $\Gamma$-invariant from~\cite{EKO22} to compute
	\begin{align*}
	\Gamma(\xi',\mathfrak t)-\Gamma(\xi,\mathfrak t)=\frac{1}{2} \left(\sum_{i=0}^n r_i\mu_i+\sum_{j\in J}(Q'\mu)_j\right)-\frac{1}{2} \left(\sum_{i=1}^n r_i\mu_i+\sum_{j\in J}(Q\mu)_j\right)=0,
	\end{align*}
	where $r_i$ denotes the rotation number and $\mu_i$ the meridian of $L_i$, $Q'$ and $Q$ denotes the linking matrices of $L_0\cup L$ and $L$, and all non-canceling terms are multiples of $\mu_0=0$.
\end{proof}

We now turn to the proof of Theorem~\ref{thm:lk}. For the proof, we need the following lemma.

\begin{lem}\label{lem:move}
Let $L$ be a Legendrian knot in some contact $3$-manifold $(M,\xi)$ and $U$ be a Legendrian meridian of $L$ with $\tb=-1$. Let $L_\pm$ denote the positive/negative stabilization of $L$ (similarly for $U$). Then contact $(-1)$-surgery along $L$ followed by contact $(+1)$-surgery along $U_\pm$ is contactomorphic to the contact $(+1)$-surgery along $L_\pm$ (see Figure~\ref{fig:Move}), i.e. $$L(-1)\cup U_\pm(+1)=L_\pm(+1).$$
\end{lem}
\begin{figure}[htbp] 
	\centering
	\includegraphics[width=7cm]{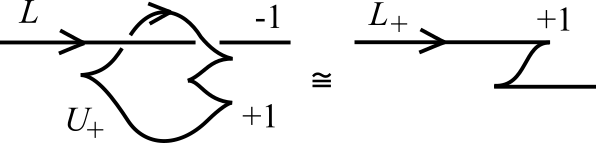}
	\caption{From a contact $(-1)$-surgery to a contact $(+1)$-surgery.}
	\label{fig:Move}
\end{figure}	

\begin{proof} We relate the two contact surgery diagrams in Figure~\ref{fig:Move} by first performing two handle slides~\cite{DG09,Av13,CEK21} followed by a lantern destabilization~\cite{LS11}, cf.~\cite{EKO22}. We first slide the red knot $L$ over $U_+$ as in Figure~\ref{fig:Move_Proof}$(i)$. Then, we slide $U_+$  over the red one as in Figure~\ref{fig:Move_Proof}$(ii)$, where the handle slides are indicated via the arrows. After isotoping Figure~\ref{fig:Move_Proof}$(iii)$ to get Figure~\ref{fig:Move_Proof}$(iv)$, we apply the lantern destabilization to get Figure~\ref{fig:Move_Proof}$(v)$.
\end{proof}

\begin{figure}[htbp] 
\centering
\includegraphics[width=9cm]{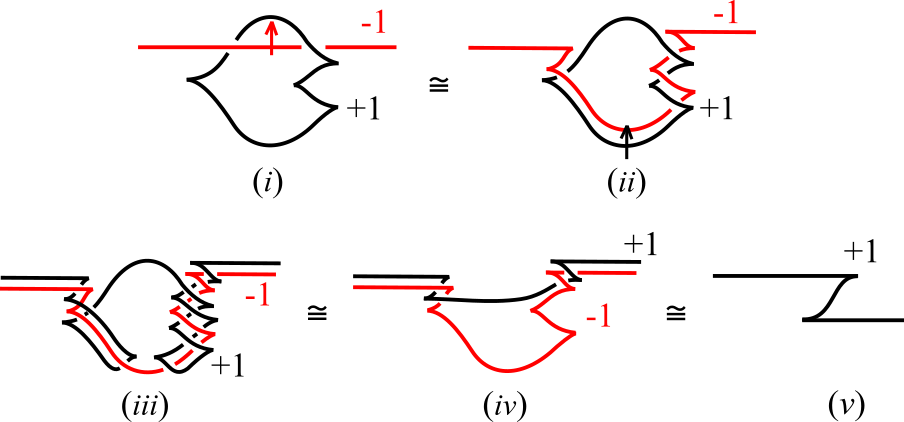}
\caption{The proof of Lemma~\ref{lem:move} via two handle slides and a lantern destabilization.}
\label{fig:Move_Proof}
\end{figure}	

\begin{proof}[Proof of Theorem~\ref{thm:lk}] Let $(N,\xi_N)$ be a contact manifold in the link of $(M,\xi)$. We construct a path of length one or two in $\lk(M,\xi)$ from $(N,\xi_N)$ to $(M,\xi)\#(S^3,\xi_1)$. (We recall that $(M,\xi)\#(S^3,\xi_1)$ can be obtained from $(M,\xi)$ by a single contact $(+1)$-surgery along a Legendrian unknot with $\tb=-2$ in a standard Darboux ball in $(M,\xi)$.) 
	
	First, we consider the case that we can obtain $(N,\xi_N)$ by a contact $(+1)$-surgery along a Legendrian knot $K$ in $(M,\xi)$. In~\cite{Av13} it is shown that performing a contact $(+1)$-surgery $K$ followed by a contact $(+1)$-surgery along a Legendrian meridian $U$ of $K$ with $\tb=-1$ corresponds to a negative stabilization of $(M,\xi)$ and thus yields $(M,\xi)\#(S^3,\xi_1)$.
	
	The case that $(N,\xi_N)$ arises as contact $(-1)$-surgery from $(M,\xi)$ can be reduced to the first case by applying  Lemma~\ref{lem:move} once.
\end{proof}

The proof of Theorem~\ref{thm:lk} directly implies the following corollary.

\begin{cor}
	(1) If $(N,\xi)$ is in the link of $(M,\xi)$. Then 
	\begin{equation*}
	d_{\lk(M,\xi)}\big((N,\xi),(M,\xi)\#(S^3,\xi_1)\big)\leq2.
	\end{equation*}
	(2) If $(N,\xi)$ can be obtained from $(M,\xi)$ by a single contact $(+1)$-surgery, then $(N,\xi)$ can be obtained from $(M,\xi)\#(S^3,\xi_1)$ by a single contact $(-1)$-surgery.
\end{cor}



\end{document}